\documentclass[a4paper,12pt]{article}

\usepackage {amsfonts, amssymb, amscd,amsthm, amsmath, eucal}

\usepackage{amsbsy}
\usepackage[all]{xy}

\usepackage{color}

\theoremstyle{plain} {
  \newtheorem{thm}{Theorem}[section]
  \newtheorem{defn}[thm]{Definition}
  \newtheorem{cor}[thm]{Corollary}
  \newtheorem{lem}[thm]{Lemma}
  \newtheorem{prop}[thm]{Proposition}
  \theoremstyle{definition}
  \newtheorem{rem}[thm]{Remark}
    \newtheorem{constr}[thm]{Construction}
  \theoremstyle{plain}
  \newtheorem{clm}[thm]{Claim}

  \newtheorem{sssection}[thm]{}
}

{

}
\renewcommand{\subsubsection}{\sssection\rm}

\newcommand{\bG}{\mathbf G}

\renewcommand{\P}{\mathbb P}

\newcommand{\can}{\text{\rm can}}

\newcommand{\id}{\text{\rm id}}
\newcommand{\pr}{\text{\rm pr}}

\newcommand{\inc}{\text{\rm inc}}

\newcommand{\const}{\text{\rm const}}
\newcommand{\Spec}{\text{\rm Spec}}

\newcommand{\Aff}{\mathbf {A}}
\newcommand{\Pro}{\mathbf {P}}

\newcommand \xra {\xrightarrow }

\newcommand \hra {\hookrightarrow }

\newcommand{\ttf}{{\text{f}}}

\renewcommand{\P}{\mathbb P}

\newcommand\mydim{\text{\rm dim}}

\renewcommand \id{\operatorname{id}}

\renewcommand \phi\varphi

\newcommand{\et}{\text{\rm\'et}}

\newcommand{\ZZ}{\mathbb Z}

\begin{document}

\title{Two purity theorems and the Grothendieck--Serre's conjecture concerning
principal $G$-bundles
}

\author{Ivan Panin\footnote{The author acknowledges support of the
RNF-grant 14-11-00456.}
}


\maketitle

\begin{abstract}
In a series of papers \cite{Pan0}, \cite{Pan1}, \cite{Pan2}, \cite{Pan3} we give a detailed and better structured
proof of the Grothendieck--Serre's conjecture
for semi-local regular rings containing a finite field. The outline of the proof
is the same as in \cite{P1},\cite{P2},\cite{P3}.
If the semi-local regular ring contains an infinite field,
then the conjecture is proved in \cite{FP}. {\it Thus the conjecture
is true for regular local rings containing a field.
}

A proof of Grothendieck--Serre conjecture on principal bundles over a semi-local regular
ring containing an arbitrary field is given in \cite{P3}. That proof is {\it heavily based}
on Theorem \ref{MainThmGeometric} stated below in the Introduction and proven in the present paper.

Theorem
\ref{MainThmGeometric}
itself is a consequence of {\it two purity theorems}
\ref{TheoremA}
and
\ref{TheoremA1}
which are of completely independent interest and which are
proven below in the present paper.
For the case when the semi-local regular
ring contains an infinite field those purity theorems
are proved in \cite{Pa2}.



\end{abstract}

\section{Main results}\label{Introduction}
Recall
\cite[Exp.~XIX, Defn.2.7]{D-G}
that an $R$-group scheme $G$ is called reductive
(respectively, semi-simple; respectively, simple), if it is affine and smooth
as an $R$-scheme and if, moreover, for each ring homomorphism
$s:R\to\Omega(s)$ to an algebraically closed field $\Omega(s)$,
its scalar extension $G_{\Omega(s)}$
is connected and reductive (respectively, semi-simple; respectively, simple) algebraic
group over $\Omega(s)$.
{\it Stress that all the groups $G_{\Omega(s)}$ are connected}.
The class of reductive group schemes
contains the class of semi-simple group schemes which in turn
contains the class of simple group schemes. This notion of a
simple $R$-group scheme coincides with the notion of a {simple semi-simple
$R$-group scheme from Demazure---Grothendieck
\cite[Exp.~XIX, Defn.~2.7 and Exp.~XXIV, 5.3]{D-G}.} {\it
Throughout the paper $R$ denotes an integral noetherian domain and $G$
denotes a reductive $R$-group scheme, unless explicitly stated
otherwise}.
\par
A well-known conjecture due to J.-P.~Serre and A.~Grothendieck
\cite[Remarque, p.31]{Se},
\cite[Remarque 3, p.26-27]{Gr1},
and
\cite[Remarque 1.11.a]{Gr2}
asserts that given a regular local ring $R$ and its field of
fractions $K$ and given a reductive group scheme $G$ over $R$ the
map
$$ H^1_{\text{\'et}}(R,G)\to H^1_{\text{\'et}}(K,G), $$
\noindent
induced by the inclusion of $R$ into $K$, has trivial kernel.

A proof of Grothendieck--Serre conjecture on principal bundles over a semi-local regular
ring containing a finite field is given in \cite{Pan3}. That proof is {\it heavily based}
on Theorem \ref{MainThmGeometric} stated below in the Introduction and proven in the present paper.

Theorem
\ref{MainThmGeometric}
itself is a consequence of {\it two purity theorems}
\ref{TheoremA}
and
\ref{TheoremA1}
which are of completely independent interest and which are
proven below in the present paper.
\begin{thm}[Theorem A]
\label{TheoremA}
Let $k$ be a field.
Let $\mathcal O$ be the semi-local ring of finitely many closed points
on a $k$-smooth irreducible affine $k$-variety $X$.
Let $K=k(X)$.
Let
$$\mu: \mathbf G \to \mathbf C$$
be a smooth $\mathcal O$-morphism of reductive
$\mathcal O$-group schemes, with a torus $\mathbf C$.
Suppose additionally that
the kernel of $\mu$ is a reductive $\mathcal O$-group scheme.
Then the following sequence
\begin{equation}
\label{Aus_Buks_sequence}
 \mathbf C(\mathcal O)/\mu(\mathbf G(\mathcal O)) \to
\mathbf C(K)/\mu(\mathbf G(K)) \xrightarrow{\sum res_{\mathfrak p}} \bigoplus_{\mathfrak p}
\mathbf C(K)/[\mathbf C(\mathcal O_{\mathfrak p})\cdot \mu(\mathbf G(K))]
\end{equation}
is exact,
where $\mathfrak p$ runs over
the height $1$ primes of $\mathcal O$ and $res_{\mathfrak p}$ is the natural map (the projection to the factor group).

\end{thm}
Let $\mathcal O$ and $K$ be as in Theorem \ref{TheoremA}.
Let $G$ be a semi-simple $\mathcal O$-group scheme.
Let
$i: Z \hra G$ be a closed subgroup scheme of the center $Cent(G)$.
{\it It is known that $Z$ is of multiplicative type}.
Let $G'=G/Z$ be the factor group,
$\pi: G \to G'$ be the projection.
It is known that $\pi$ is finite surjective and strictly flat. Thus
the sequence of $\mathcal O$-group schemes
\begin{equation}
\label{ZandGndGprime}
\{1\} \to Z \xra{i} G \xra{\pi} G^{\prime} \to \{1\}
\end{equation}
induces an exact sequence of group sheaves in $\text{fppt}$-topology.
Thus for every $\mathcal O$-algebra $R$ the sequence
(\ref{ZandGndGprime})
gives rise to a boundary operator
\begin{equation}
\label{boundary}
\delta_{\pi,R}: G'(R) \to \textrm{H}^1_{\text{fppt}}(R,Z)
\end{equation}
One can check that it is a group homomorphism
(compare \cite[Ch.II, \S 5.6, Cor.2]{Se}).
Set
\begin{equation}
\label{AnotherFunctor}
{\cal F}(R)= \textrm{H}^1_{\text{fttp}}(R,Z)/ Im(\delta_{\pi,R}).
\end{equation}

\begin{thm}
\label{TheoremA1}
Let $k$, $\mathcal O$ and $K$ be as in Theorem \ref{TheoremA}.
The following sequence
\begin{equation}
\label{Aus_Buks_sequence_2}
 \frac{\textrm{H}^1_{\text{fttp}}(\mathcal O,Z)}{ Im(\delta_{\pi,\mathcal O})} \to
\frac{\textrm{H}^1_{\text{fttp}}(K,Z)}{Im(\delta_{\pi,K})}  \xrightarrow{\sum res_{\mathfrak p}} \bigoplus_{\mathfrak p}
\frac{\textrm{H}^1_{\text{fttp}}(K,Z)}{[\textrm{H}^1_{\text{fttp}}(\mathcal O,Z)+ Im(\delta_{\pi,K})]}
\end{equation}
is exact.
\end{thm}
These two theorems
yield the following result.
\begin{thm}\label{MainThmGeometric}
Let $k$ be a field.
Let $\mathcal O$ be the semi-local ring of finitely many closed points
on a $k$-smooth irreducible affine $k$-variety $X$.
Let $K=k(X)$.
Assume that
for all semi-simple simply connected reductive $\mathcal O$-group schemes $H$
the pointed set map
$$H^1_{\text{\rm\'et}}(\mathcal O, H) \to H^1_{\text{\rm\'et}}(k(X), H),$$
\noindent
induced by the inclusion of $\mathcal O$ into its fraction field $k(X)$, has trivial kernel.
Then for
any reductive $\mathcal O$-group scheme $G$
the pointed set map
$$H^1_{\text{\rm\'et}}(\mathcal O, G) \to H^1_{\text{\rm\'et}}(K, G),$$
\noindent
induced by the inclusion of $\mathcal O$ into its fraction field $K$, has trivial kernel.
\end{thm}

\begin{rem}
The proof of the latter theorem is subdivided into two steps.
Firstly, given a semi-simle
$\mathcal O$-group scheme $G$
we prove that the Grothendieck--Serre conjecture holds for $\mathcal O$-group scheme $G$
providing it holds for its simply-connected cover
$G^{sc}$ and all inner forms of that
simply-connected
$\mathcal O$-group scheme
$G^{sc}$.

Secondly, given a reductive $\mathcal O$-group scheme $G$
we prove that
the Grothendieck--Serre conjecture holds for $\mathcal O$-group scheme $G$
providing it holds for
the derived $\mathcal O$-group scheme $G_{der}$ of $G$
and for all inner forms of that
derived $\mathcal O$-group scheme $G_{der}$ of $G$.
\end{rem}



After the first articles
\cite{C-T/P/S}
and
\cite{R}
on purity theorems for algebraic groups,
various versions of purity theorems were proved in
\cite{C-TO},
\cite{PS},
\cite{Z},
\cite{Pan2}.
The most general result in the so called constant case was given in
\cite[Exm.3.3]{Z}. This result follows now from our Theorem (A) by
 taking $G$ to be a $k$-rational reductive group, $C=\mathbb G_{m,k}$
and
$\mu: G \to \mathbb G_{m,k}$ a dominant $k$-group morphism. The papers
\cite{PS},
\cite{Z},
\cite{Pan2}
contain results for the nonconstant case. However they only consider specific examples
of algebraic scheme morphisms $\mu: G \to C$.

It seems plausible to expect a purity theorem in the following context.
Let $R$ be a regular local ring. Let
$\mu: G \to T$ be a smooth morphism of reductive $R$-group schemes
with an $R$-torus $T$.
Let  ${\cal F}$ be the covariant functor from the category of
commutative rings to the category of abelian groups
given by
$S \mapsto T(S)/\mu(G(S))$. Then ${\cal F}$ should satisfies purity for $R$.

Let us to point out that
we use transfers for the functor
$R \mapsto C(R)$, but {\it we do not use at all the norm principle}
for the homomorphism $\mu: G \to C$.
It is a big open question: whether the morphism $\mu$ from Theorem
\ref{TheoremA}
satisfies the norm principle even for finite separable field extensions.
So, we are not able to say that the functor $\mathcal F$ is a presheaf
with transfers in the sense of Voevodsky or in any other weaker sense. {\it That is the major trouble
for any attempt to prove Theorem
\ref{TheoremA}}.

{\it Here is our approach}.
Given an element $\xi \in C(K)$ such that its image
$\bar \xi \in \mathcal F(K)$ is an
$\mathcal O$-unramified element
one can find a rather refined finite correspondence of the form
\begin{equation}\label{finite correspondence_1}
\Aff^1_{\mathcal O} \xleftarrow{\sigma} \mathcal X^{\prime} \xrightarrow{q^{\prime}_X} X
\end{equation}
(see the diagram (\ref{DeformationDiagram0}))
and use it in the "constant group" case to write down a good candidate
$\bar \xi_{\mathcal O} \in \mathcal F(\mathcal O)$
(see (\ref{U_alpha}))
for a lift of the element
$\bar \xi$ to $\mathcal F(\mathcal O)$.
In general, since the group-scheme $G$ does not come from the ground field
we need to equate two its pull-backs
$(pr_{\mathcal O} \circ \sigma)^*(G)$
and
$(q^{\prime}_{X})^*(G)$
over
$\mathcal X^{\prime}$.
Here $pr_{\mathcal O}: \Aff^1_{\mathcal O} \to \text{Spec}(\mathcal O)$
is the projection.
The same we need to do with the torus $C$.
Due to these requirements our
construction of the finite correspondence
and of a good candidate
$\bar \xi_{\mathcal O} \in \mathcal F(\mathcal O)$
(see (\ref{U_alpha_2}))
is quite involved.
It is especially quite involved in the case of the finite base field.
Surely, we use Bertini type results which are due to
Poonen. Even this does not resolve all
the difficulties.
Nevertheles there is a construction of the desired finite correspondence as
for the "constant group" case, so for the general case.
It is done below in Theorem
\ref{equating3}.

The finite surjective morphism $\sigma$ of the $\mathcal O$-schemes
has the following property: for the corresponding fraction field extension
$K(u) \subset \mathcal K$
the element
$$\zeta_u:=N_{\mathcal K/K(u)}((q^{\prime}_X)^*(\xi)) \in C(K(u))$$
is such that its image
$\bar \zeta_u \in \mathcal F(K(u))$
is $K[u]$-unramified.
The latter yields that the element
$\bar \zeta_u$ is constant that is it belongs to
$\mathcal F(K)$. Thus the evaluations of
$\bar \zeta_u$ at $u=0$ and at $u=1$ coincide.
Now simple computations show that the element
$\bar \xi_{\mathcal O} \in \mathcal F(\mathcal O)$
is indeed a lift of the element
$\bar \xi \in \mathcal F(K)$.
Details are given in Section
\ref{Section_Aus_Buksbaum}.

The article is organized as follows.
In Sections \ref{NiceTriplesAndGr_Schemes} and \ref{First Application}
geometric results from \cite{P} are used to prove stronger versions
of some results from 
\cite{Pan1}.
In Section \ref{SectNorms} the construction of norm maps is recalled.
In Section \ref{SectUnramifiedElements} groups of unramified elements are defined
and it is proved that in certain cases the norm map takes
a specific unramified element to unramified one
(see Lemma \ref{KeyUnramifiedness}).
In Section
\ref{SectSpecializationMaps}
specialization maps are defined
and a homotopy invariance theorem for the group of unramified elements is proved
(see Theorem \ref{HomInvNonram}).
In Section \ref{Section_Aus_Buksbaum}
Theorem \ref{TheoremA} is proved.
In Section \ref{OneMorePurity}
Theorem \ref{TheoremA1} is proved.
Finally, in section \ref{sec:MainThmGeometric}
Theorem \ref{MainThmGeometric} is proved.


The author thanks A.~Suslin for his interest in the topic of the present article. He
also thanks to A.Stavrova for paying his attention to Poonen's works on Bertini type theorems
for varieties over finite fields. He thanks D.Orlov for useful comments concerning
the weighted projective spaces tacitely involved in the construction of elementary fibrations.
He thanks M.Ojanguren for many inspiring ideas arising from our joint works with him.


\section{Equating group scheme morphisms}
\label{EqGroupSchemeMorph}

Let $S$ be a regular semi-local
{\it irreducible scheme.}
Let
$\mu_1: G_1\to C_1$ and
$\mu_2: G_2\to C_2$
be two smooth $S$-group scheme morphisms with tori $C_1$ and $C_2$.
Suppose that $G_1$ and $G_2$ are reductive $S$-group schemes
which are forms of each other and suppose that $C_1$ and $C_2$ are forms of each other.
Let $T \subset S$ be a connected non-empty closed
sub-scheme of $S$, and
$\varphi: G_1|_T \to G_2|_T$,
$\psi: C_1|_T \to C_2|_T$
be $T$-group scheme isomorphisms.
By Theorem
\cite[Thm. 3.1]{Pan1}
there exists a finite \'{e}tale morphism
$\tilde S \xra{\pi} S$
with an {\it irreducible scheme} $\tilde S$ and
a section $\delta: T \to \tilde S$ of $\pi$ over $T$ and $\tilde S$-group scheme isomorphisms
$$\Phi: G_{1, \tilde S} \to G_{2, \tilde S} \ \ \text{and} \
\Psi: C_{1, \tilde S} \to C_{2, \tilde S}$$
such that
$\delta^*(\Phi)= \varphi$, $\delta^*(\Psi)= \psi$.
For those $\tilde S$, $T$, $\delta$, $\Phi$ and $\Psi$
the following result holds.

\begin{thm}
\label{PropEquatingGroups_1_1_1}
Let $S$ be the regular semi-local
{\it irreducible scheme.}
Let $T \subset S$ be the connected non-empty closed
sub-scheme of $S$.
Suppose that the morphisms $\phi$ and $\psi$, $(\mu_1|_{T})$ and $(\mu_2|_{T})$ are such that
the diagram commutes
\begin{equation}
\label{phi_psi_mu_1_mu_2}
    \xymatrix{
     G_1|_T  \ar[rr]^{\phi} \ar[d]_{\mu_1|_T} &&
     G_2|_T \ar[d]^{\mu_2|_T}  &\\
     C_1|_T \ar[rr]^{\psi} && C_2|_T. &\\
    }
\end{equation}
Then
$\mu_{2, \tilde S} \circ \Phi = \Psi \circ \mu_{1, \tilde S} : G_{1,\tilde S} \to C_{2,\tilde S}$
as the $\tilde S$-group scheme morphisms.
\end{thm}

\begin{proof}
Recall that
$\mu_r$ can be naturally presented as a composition
$$G_r \xrightarrow{can_r} Corad(G_r) \xrightarrow{{\bar \mu}_r} C_r.$$
Since
$can_{2, \tilde S} \circ \Phi = Corad(\Phi) \circ can_{1, \tilde S}$
it remains to check that
${\bar \mu}_{2, \tilde S} \circ Corad(\Phi)=\Psi \circ {\bar \mu}_{1, \tilde S}$.

The equality
$(\mu_2|_{T}) \circ \varphi = \psi \circ (\mu_1|_{T})$
holds
by the assumption of the Theorem. It yields
the equality
$({\bar \mu_2}|_{T}) \circ Corad(\varphi) = \psi \circ ({\bar \mu_1}|_{T})$.
The equality
${\bar \mu}_{2, \tilde S} \circ Corad(\Phi)=\Psi \circ {\bar \mu}_{1, \tilde S}$
follows now from
\cite[Prop.3.5]{Pan1}
since $\tilde S$ is irreducible.
Whence the theorem.
\end{proof}

\section{Nice triples and group scheme morphisms}
\label{NiceTriplesAndGr_Schemes}
See \cite[Defn.3.1]{PSV} for the definition of a nice triple and
see \cite[Defn.3.1]{PSV} for the definition of a morphism
between nice triples.
Those definitions are reproduced in
\cite[Defn.3.1 and 3.3]{P}.
The notion of a special nice triple
is given in \cite[Defn.3.4]{P}.
We need in an extension of theorems
\cite[Thm. 3.9]{P}, \cite[Thm. 3.9]{Pan1}.

For that it is convenient to give two definitions
under the following set up.
Let $k$ be a field and $\mathcal O$ be
{\it the semi-local ring of finitely many closed points
on a $k$-smooth irreducible affine $k$-variety $X$.
}
Let $U=\text{Spec}(\mathcal O)$.
Let $(\mathcal X,f,\Delta)$ be a special nice triple
over $U$ and let
$G_{\mathcal X}$ be a reductive
$\mathcal X$-group scheme and
$G_U:= \Delta^*(G_{\mathcal X})$
and
$G_{\text{const}}:=q^*_U(G_U)$.
Let
$\theta: (q^{\prime}: \mathcal X^{\prime} \to U, f^{\prime}, \Delta^{\prime}) \to (q: \mathcal X \to U, f, \Delta)$
be a morphism between nice triples over $U$.
The following definition is from \cite[Defn. 4.1]{Pan1}
\begin{defn}{\rm
\label{def:q_constant}
We say that
the morphism $\theta$ {\it equates}
the reductive
$\mathcal X$-group schemes
$G_{\mathcal X}$ and $G_{\text{const}}$, if
there is an
$\mathcal X^{\prime}$-group scheme isomorphism
$\Phi: \theta^*(G_{\text{const}}) \to \theta^*(G_{\mathcal X})$
with
$(\Delta^{\prime})^*(\Phi)= id_{G_U}$.
}
\end{defn}
Further, let $C_{\mathcal X}$ be an
$\mathcal X$-tori and
$C_U:= \Delta^*(C_{\mathcal X})$
and
$C_{\text{const}}:=q^*_U(\mathbf C_U)$.
Let
$\mu_{\mathcal X}: G_{\mathcal X} \to C_{\mathcal X}$
be an $\mathcal X$-group scheme morphism smooth as a scheme morphism.
Let
$\mu_U = \Delta^*(\mu_{\mathcal X})$
and
$\mu_{\const}: G_{\const} \to C_{\const}$
be the the pull-back of $\mu_U$ to
$\mathcal X$.

\begin{defn}[Equating morphisms]
\label{def:q_constant_mu}{\rm
We say that
the morphism $\theta$ {\it equates}
the reductive
$\mathcal X$-group scheme morphisms
$\mu_{\const}$
and
$\mu_{\mathcal X}$
if there are
$\mathcal X^{\prime}$-group schemes isomorphisms
$$\Phi: \theta^*(G_{\text{const}}) \to \theta^*(G_{\mathcal X}) \ \ \text{and} \ \ \Psi: \theta^*(C_{\text{const}}) \to \theta^*(C_{\mathcal X})$$
with
$(\Delta^{\prime})^*(\Phi)= id_{G_U}$,
$(\Delta^{\prime})^*(\Phi)= id_{G_U}$
and
$\theta^*(\mu_{\mathcal X}) \circ \Phi = \Psi \circ \theta^*(\mu_{const})$.
Clearly, if the morphism $\theta$ {\it equates} morphisms
$\mu_{\const}$
and
$\mu_{\mathcal X}$,
then it {\it equates}
$G_{\text{const}}$ with $G_{\mathcal X}$ and $C_{\text{const}}$ with $C_{\mathcal X}$.
}
\end{defn}

\begin{rem}\label{rem:theta_circ_theta'}
Let $\rho: (\mathcal X'',f'',\Delta'') \to (\mathcal X',f',\Delta')$ and $\theta: (\mathcal X',f',\Delta') \to (\mathcal X,f,\Delta)$
be morphisms of nice triples over $U$.
If $\theta$ equates $\mu_{\text{const}}$ with $\mu_{\mathcal X}$,
then $\theta \circ \rho$ equates
$\mu_{\text{const}}$ and $\mu_{\mathcal X}$ also.
\end{rem}

\begin{thm}
\label{equating3_triples}
Let $U$ be as above in this section. Let $(\mathcal X,f,\Delta)$ be a special nice triple
over $U$.
Let $G_{\mathcal X}$ be a reductive
$\mathcal X$-group scheme and
$G_U:= \Delta^*(G_{\mathcal X})$
and
$G_{\text{const}}:=q^*_U(G_U)$.
Let $C_{\mathcal X}$ be an
$\mathcal X$-tori and
$C_U:= \Delta^*(C_{\mathcal X})$
and
$C_{\text{const}}:=q^*_U(C_U)$.
Let
$\mu_{\mathcal X}: G_{\mathcal X} \to C_{\mathcal X}$
be an $\mathcal X$-group scheme morphism smooth as a scheme morphism.
Let
$\mu_U = \Delta^*(\mu_{\mathcal X})$
and
$\mu_{\const}: G_{\const} \to C_{\const}$
be the the pull-back of $\mu_U$ to
$\mathcal X$.
Then
there exist a morphism
$\theta'': (q^{\prime\prime}: \mathcal X^{\prime\prime} \to U, f^{\prime\prime}, \Delta^{\prime\prime}) \to
(q: \mathcal X \to U, f, \Delta)$
between nice triples over $U$ such that
\begin{itemize}
\item[(i)]
the morphism $\theta''$ equates the reductive $\mathcal X$-group scheme morphisms
$\mu_{\const}$ and $\mu_{\mathcal X}$;
\item[(ii)]
the triple
$(\mathcal X^{\prime\prime},f^{\prime\prime},\Delta^{\prime\prime})$
is a special nice triple
over $U$ subjecting to the conditions
$(1^*)$ and $(2^*)$ from \cite[Defn. 3.7]{P}.
\end{itemize}
\end{thm}

\begin{proof}[Proof of Theorem \ref{equating3_triples}]
Let $U$ be as in the theorem. Let
$(\mathcal X,f,\Delta)$
be a special nice triple
over $U$ as in the theorem.
By the definition
of a nice triple there exists a finite surjective morphism
$\Pi:\mathcal X\to\Aff^1\times U$ of $U$-schemes.
The construction \cite[Constr. 4.2]{P}
gives us now the data
$(\mathcal Z,\mathcal Y, S, T)$,
where $\mathcal Z, \mathcal Y, T$
are closed subsets as in $\mathcal X$ finite over $U$.
Particularly,  they are semi-local. If
$y_1,...,y_n$ are all closed points of $Y$,
then $S=\text{Spec}(\mathcal O_{X,y_1,...,y_n})$.

Further, let $G_U=\Delta^*(G_{\mathcal
X})$ be as in the hypotheses of Theorem
\ref{equating3_triples} and
let
$G_{\const}$
be the pull-back of
$G_U$ to $\mathcal X$. Finally,
let
$\varphi:G_{\const}|_T \to G_{\mathcal X}|_T$
be the canonical
isomorphism. Recall that by assumption $\mathcal X$ is $U$-smooth and irreducible,
and thus $S$ is regular and irreducible.
By
\cite[Thm. 3.1]{Pan1}
and Theorem
\ref{PropEquatingGroups_1_1_1}
there exists a finite
\'etale morphism $\theta_0:S^{\prime}\to S$, a section
$\delta:T\to S^{\prime}$ of $\theta_0$ over $T$ and isomorphisms
$\Phi_0:\theta^*_0(G_{\const,S})\to\theta^*_0(G_{\mathcal X}|_S)$
and
$\Psi_0:\theta^*_0(C_{\const,S})\to\theta^*_0(C_{\mathcal X}|_S)$
such that
$\delta^*(\Phi_0)=\varphi$,
$\delta^*(\Psi_0)=\psi$
and
\begin{equation}
\label{mu_X_const_0}
\theta^*_0(\mu_{\mathcal X}|_S) \circ \Phi_0 = \Psi_0 \circ \theta^*_0(\mu_{\const,S}): \theta^*_0(G_{\const,S}) \to \theta^*_0(C_{\mathcal X}|_S)
\end{equation}
{\it where the scheme $S^{\prime}$ is irreducible.
}
Consider now the diagram (4) from the construction
\cite[Constr. 4.2]{P}.
We may and will now suppose that
the neighborhood $\mathcal V$ of the points
$\{y_1,\dots,y_n\}$
from that diagram is chosen such that
there are $\mathcal V'$-group schemes isomorphisms
$\Phi: \theta^*(G_{\const,\mathcal V})\to\theta^*(G_{\mathcal X}|_{\mathcal V})$
and
$\Psi: \theta^*(C_{\const,\mathcal V})\to\theta^*(C_{\mathcal X}|_{\mathcal V})$
with $\Phi|_{S'}=\Phi_0$ and $\Psi|_{S'}=\Psi_0$.
Clearly,
$\delta^*(\Phi)=\varphi$
and
$\delta^*(\Psi)=\psi$.
It is less clear, however it is still true that
\begin{equation}
\label{mu_X_const_1}
\theta^*(\mu_{\mathcal X}|_{\mathcal V}) \circ \Phi = \Psi \circ \theta^*(\mu_{\const,\mathcal V}):
\theta^*(G_{\const,\mathcal V}) \to \theta^*(C_{\mathcal X}|_\mathcal V)
\end{equation}

Applying the second part of the construction
\cite[Constr. 4.2]{P} and also
the proposition
\cite[Prop. 4.3]{P}
to the finite
\'etale morphism $\theta: \mathcal V^{\prime}\to \mathcal V$ and to the section
$\delta:T\to \mathcal V^{\prime}$ of $\theta$ over $T$
we get \\
1) firstly, a triple $(\mathcal X^{\prime},f^{\prime},\Delta^{\prime})$;\\
2) secondly, the \'{e}tale morphism of $U$-schemes $\theta: \mathcal X^{\prime} \to \mathcal X$;\\
3) thirdly, inclusions of $U$-schemes $S\subset \mathcal W$ and $S' \subset \mathcal X'$.\\
Further we get \\
(i) the special nice triple $(q_U\circ \theta: \mathcal X^{\prime}\to U,f^{\prime},\Delta^{\prime})$
over $U$;\\
(ii) the morphism $\theta$ is a morphism
$(\mathcal X^{\prime},f^{\prime},\Delta^{\prime}) \to (\mathcal X,f,\Delta)$
between the nice triples and it {\it equates} the $\mathcal X$-group scheme morphisms
$\mu_{\const}$
and
$\mu_{\mathcal X}$;\\
(iii) the equality $f^{\prime}=\theta^*(f)$.

To complete the proof of the theorem just apply
the theorem \cite[Thm. 3.9]{P} to the
the special nice triple
$(\mathcal X^{\prime},f^{\prime},\Delta^{\prime})$
and use the remark
\ref{rem:theta_circ_theta'}.
\end{proof}

\section{A strong form of the theorem 5.1 from \cite{Pan1}}
\label{First Application}
Let $X$ be an affine $k$-smooth irreducible $k$-variety, and let $x_1,x_2,\dots,x_n$ be closed points in $X$.
Let $U=Spec(\mathcal O_{X,\{x_1,x_2,\dots,x_n\}})$.
Let $\bG$ be a reductive
$X$-group scheme
and let
$\bG_U= can^*(\bG)$
be the pull-back of $\bG$ to $U$.
Let $\mathbf C$ be an $X$-torus
and let
$\mathbf C_U= can^*(\mathbf C)$
be the pull-back of $\bG$ to $U$.
Let $\mu: \bG \to \mathbf C$
be an $X$-group scheme morphism which is smooth as an $X$-scheme morphism.
Let $\mu_U=can^*(\mu)$.
The following result is a strong form of the theorem 5.1 from
\cite{Pan1}.

\begin{thm}\label{equating3}
Given a non-zero function $\textrm{f}\in k[X]$ vanishing at each point $x_i$,
there is a diagram of the form
\begin{equation}
\label{DeformationDiagram0}
    \xymatrix{
\Aff^1 \times U\ar[drr]_{\pr_U}&&\mathcal X \ar[d]^{}
\ar[ll]_{\sigma}\ar[d]_{q_U}
\ar[rr]^{q_X}&&X &\\
&&U \ar[urr]_{\can}\ar@/_0.8pc/[u]_{\Delta} &\\
    }
\end{equation}
with an irreducible {\bf affine} scheme $\mathcal X$, a smooth morphism $q_U$, a finite surjective $U$-morphism $\sigma$ and an essentially smooth morphism $q_X$,
and a function $f^{\prime} \in q^*_X(\textrm{f} \ )k[\mathcal X]$,
which enjoys the following properties:
\begin{itemize}
\item[\rm{(a)}]
if
$\mathcal Z^{\prime}$ is the closed subscheme of $\mathcal X$ defined by the principal ideal
$(f^{\prime})$, the morphism
$\sigma|_{\mathcal Z^{\prime}}: \mathcal Z^{\prime} \to \Aff^1\times U$
is a closed embedding and the morphism
$q_U|_{\mathcal Z^{\prime}}: \mathcal Z^{\prime} \to U$ is finite;
\item[\rm{(a')}] $q_U\circ \Delta=id_U$ and $q_X\circ \Delta=can$ and $\sigma\circ \Delta=i_0$ \\
(the first equality shows that $\Delta(U)$ is a closed subscheme in $\mathcal X$);
\item[\rm{(b)}] $\sigma$
is \'{e}tale in a neighborhood of
$\mathcal Z^{\prime}\cup \Delta(U)$;
\item[\rm{(c)}]
$\sigma^{-1}(\sigma(\mathcal Z^{\prime}))=\mathcal Z^{\prime}\coprod \mathcal Z^{\prime\prime}$
scheme theoretically
for some closed subscheme $\mathcal Z^{\prime\prime}$
\\ and
$\mathcal Z^{\prime\prime} \cap \Delta(U)=\emptyset$;
\item[\rm{(d)}]
$\mathcal D_0:=\sigma^{-1}(\{0\} \times U)=\Delta(U)\coprod \mathcal D^{\prime}_0$
scheme theoretically
for some closed subscheme $\mathcal D^{\prime}_0$
and $\mathcal D^{\prime}_0 \cap \mathcal Z^{\prime}=\emptyset$;
\item[\rm{(e)}]
for $\mathcal D_1:=\sigma^{-1}(\{1\} \times U)$ one has
$\mathcal D_1 \cap \mathcal Z^{\prime}=\emptyset$.
\item[\rm{(f)}]
there is a monic polinomial
$h \in \mathcal O[t]$
such that
$(h)=Ker[\mathcal O[t] \xrightarrow{\sigma^*} k[\mathcal X] \xrightarrow{-} k[\mathcal X]/(f^{\prime})]$, \\
where $\mathcal O:=k[U]$ and the map bar takes any $g\in k[\mathcal X]$ to ${\bar g}\in k[\mathcal X]/(f^{\prime})$;\\
\item[\rm{(g)}] there are $\mathcal X$-group scheme isomorphisms
$\Phi:  p^*_U(\bG_U)\to p^*_X(\bG)$,
$\Psi:  p^*_U(\mathbf C_U)\to p^*_X(\mathbf C)$
with
$\Delta^*(\Phi)= id_{\bG_U}$,
$\Delta^*(\Psi)= id_{\mathbf C_U}$
and
$p^*_X(\mu) \circ \Phi=\Psi \circ p^*_U(\mu_U)$.
\end{itemize}
\end{thm}

\begin{proof}[Proof of Theorem \ref{equating3}]
By \cite[Prop. 3.6]{P} one can shrink $X$ such that
$x_1,x_2, \dots , x_n$ are still in $X$ and $X$ is affine, and then to construct a special nice triple
$(q_U: \mathcal X \to U, \Delta, f)$ over $U$ and an essentially smooth morphism $q_X: \mathcal X \to X$ such that
$q_X \circ \Delta= can$, $f=q^*_X(\text{f})$ and the set of closed points of $\Delta(U)$ is
contained in the set of closed points of $\{f=0\}$.

Set $\bG_{\mathcal X}=q^*_X(\bG)$, then $\Delta^*(\bG_{\mathcal X})=can^*(\bG)$. Thus the $U$-group scheme
$\bG_U$ from Theorem \ref{equating3_triples} and the $U$-group scheme
$\bG_U$ from Theorem \ref{equating3} are the same. By Theorem
\ref{equating3_triples}
there exists a morphism
$\theta:(\mathcal X_{new},f_{new},\Delta_{new})\to(\mathcal X,f,\Delta)$
such that the triple
$(\mathcal X_{new},f_{new},\Delta_{new})$
is a special nice triple
over $U$
subject to the conditions
$(1^*)$ and $(2^*)$ from
\cite[Defn. 3.7]{P}.
And, additionally,
there are isomorphism
$$\Phi: (q_U\circ \theta)^*(\bG_U)=\theta^*(G_{\text{const}})
\to \theta^*(G_{\mathcal X})=(q_X \circ \theta)^*(\bG) \ \text{with} \ (\Delta_{new})^*(\Phi)= id_{G_U}$$
$$\Psi: (q_U\circ \theta)^*(\mathbf C_U)=\theta^*(C_{\text{const}}) \to \theta^*(C_{\mathcal X})=(q_X \circ \theta)^*(\mathbf G)$$
of
$\mathcal X_{new}$-group schemes such that
$(\Delta^{\prime})^*(\Phi)= id_{G_U}$,
$(\Delta^{\prime})^*(\Phi)= id_{G_U}$
and
\begin{equation}
\label{PhiAndPsiRelation}
\theta^*(\mu_{\mathcal X}) \circ \Phi = \Psi \circ \theta^*(\mu_{const}).
\end{equation}

The triple
$(\mathcal X_{new},f_{new},\Delta_{new})$
is a special nice triple
{\bf over} $U$
subject to the conditions
$(1^*)$ and $(2^*)$ from Definition
\cite[Defn. 3.7]{P}.
Thus by \cite[Thm. 3.8]{P}
there is a finite surjective morphism
$\Aff^1\times U \xleftarrow{\sigma_{new}} \mathcal X_{new}$
of the $U$-schemes satisfying the conditions
$(a)$ to $(\textrm{f})$
from that Theorem. Hence one has a diagram of the form
\begin{equation}
\label{DeformationDiagram0_1}
    \xymatrix{
\Aff^1 \times U\ar[drr]_{\pr_U}&&\mathcal X_{new} \ar[d]^{}
\ar[ll]_{\sigma_{new}}\ar[d]_{q_U\circ \theta}
\ar[rr]^{q_X\circ \theta}&&X &\\
&&U \ar[urr]_{\can}\ar@/_0.8pc/[u]_{\Delta^{\prime}} &\\
    }
\end{equation}
with the irreducible scheme $\mathcal X_{new}$, the smooth morphism $q_{U,new}:=q_U\circ \theta$,
the finite surjective morphism $\sigma_{new}$ and the essentially smooth morphism $q_{X,new}:=q_X\circ \theta$
and with the function
$f_{new} \in (q_{X,new})^*(\textrm{f})k[\mathcal X_{new}]$,
which after identifying notation enjoy the properties
(a) to (\textrm{f}) from Theorem
\ref{equating3}.
The equality
(\ref{PhiAndPsiRelation})
shows that
the isomorphisms $\Phi$
and $\Psi$ subject to
the condition
(g).
Whence the Theorem \ref{equating3}.
\end{proof}

To formulate a consequence of the theorem
\ref{equating3} (see Corollary \ref{ElementaryNisSquareNew_1}),
note that using the items (b) and (c) of Theorem
\ref{equating3}
one can find an element
$g \in I(\mathcal Z^{\prime\prime})$
such that \\
(1) $(f^{\prime})+(g)=\Gamma(\mathcal X, \mathcal O_{\mathcal X})$, \\
(2) $Ker((\Delta)^*)+(g)=\Gamma(\mathcal X, \mathcal O_{\mathcal X})$, \\
(3) $\sigma_g=\sigma|_{\mathcal X_g}: \mathcal X_g \to \Aff^1_U$ is \'{e}tale.\\

Here is the corollary. It is proved in \cite[Cor. 7.2]{P}.
\begin{cor}
\label{ElementaryNisSquareNew_1}
The function $f^{\prime}$ from Theorem \ref{equating3}, the polinomial $h$ from the item $(\textrm{f} \ )$
of that Theorem, the morphism $\sigma: \mathcal X \to \Aff^1_U$
and the function
$g \in \Gamma(\mathcal X,\mathcal O_{\mathcal X} )$
defined just above
enjoy the following properties:
\begin{itemize}
\item[\rm{(i)}]
the morphism
$\sigma_g= \sigma|_{\mathcal X_g}: \mathcal X_g \to \Aff^1\times U $
is \'{e}tale,
\item[\rm{(ii)}]
data
$ (\mathcal O[t],\sigma^*_g: \mathcal O[t] \to \Gamma(\mathcal X,\mathcal O_{\mathcal X})_g, h ) $
satisfies the hypotheses of
\cite[Prop.2.6]{C-TO},
i.e.
$\Gamma(\mathcal X,\mathcal O_{\mathcal X} )_g$
is a finitely generated
$\mathcal O[t]$-algebra, the element $(\sigma_g)^*(h)$
is not a zero-divisor in
$\Gamma(\mathcal X,\mathcal O_{\mathcal X} )_g$
and
$\mathcal O[t]/(h)=\Gamma(\mathcal X,\mathcal O_{\mathcal X})_g/h\Gamma(\mathcal X,\mathcal O_{\mathcal X})_g$ \ ,
\item[\rm{(iii)}]
$(\Delta(U) \cup \mathcal Z') \subset \mathcal X_g$ \ and $\sigma_g \circ \Delta=i_0: U\to \Aff^1\times U$,
\item[\rm{(iv)}]
$\mathcal X_{gh} \subseteq \mathcal X_{gf^{\prime}}\subseteq \mathcal X_{f^{\prime}}\subseteq \mathcal X_{q^*_X(\textrm{f})}$ \ ,
\item[\rm{(v)}]
$\mathcal O[t]/(h)=\Gamma(\mathcal X,\mathcal O_{\mathcal X})/(f^{\prime})$
and
$h\Gamma(\mathcal X,\mathcal O_{\mathcal X})=(f^{\prime})\cap I(\mathcal Z^{\prime\prime})$
and
$(f^{\prime}) +I(\mathcal Z^{\prime\prime})=\Gamma(\mathcal X,\mathcal O_{\mathcal X})$.
\end{itemize}
\end{cor}

\section{Norms}
\label{SectNorms}
In sections 12, 13, 14 we prove few results which are used below to prove Theorems
\ref{Aus_Buksbaum}, \ref{PurityForSubgroup}
and
\ref{MainThmGeometric}.

Let $k\subset K\subset L$ be field extensions and assume that $L$
is finite separable over $K$. Let $K^{sep}$ be a separable closure
of $K$ and $\sigma_i:K\to K^{sep},\enspace 1\leq i\leq n$
the different embeddings of $K$ into $L$. Let $C$ be a $k$-smooth
commutative algebraic group scheme defined over $k$. One can define
a norm map
$${\mathcal N}_{L/K}:C(L)\to C(K)$$
by
${\mathcal N}_{L/K}(\alpha)=\prod_i C(\sigma_i)(\alpha) \in C(K^{sep})^{{\mathcal G}(K)} =C(K)\ .
$
In \cite{Pa2} following Suslin and Voevodsky
\cite[Sect.6]{SV} we
generalize this construction to finite flat ring extensions.
\smallskip
Let $p:X\to Y$ be a finite flat morphism of affine schemes.
Suppose that its rank is constant, equal to $d$. Denote by
$S^d(X/Y)$ the $d$-th symmetric power of $X$ over $Y$.

\smallskip
Let $k$ be a field. Let $\mathcal O$ be the semi-local ring of finitely many {\bf closed} points
on a smooth affine irreducible $k$-variety $X$.
Let $C$ be an affine smooth commutative $\mathcal O$-group scheme,
Let $p:X\to Y$ be a finite flat morphism of affine $\mathcal O$-schemes
and $f:X\to C$ any $\mathcal O$-morphism. In \cite{Pa2} {\it the norm} $N_{X/Y}(f)$ of $f$ {\it is defined as the composite map}
\begin{equation}
\label{NormMap}
Y \xra{N_{X/Y}}
S^d(X/Y) \to S^d_{\mathcal O}(X) \xra{S^d_{\mathcal O}(f)} S^d_{\mathcal O}(C)\xra{\times} C
\end{equation}
Here we write $"\times"$ for the group law on $C$.
The norm maps $N_{X/Y}$ satisfy the following conditions
\begin{itemize}
\item[(i')]
Base change: for any map $f:Y'\to Y$ of affine schemes, putting
$X'=X\times_Y Y'$ we have a commutative diagram
$$
\begin{CD}
C(X)                  @>{(id \times f)^{*}}>>            C(X^{\prime})      \\
@V{N_{X/Y}}VV @VV{N_{X^{\prime}/Y^{\prime}}}V    \\
C(Y)                  @>{f^{*}}>>            C(Y^{\prime})
\end{CD}
$$
\item[(ii')]
multiplicativity: if $X=X_1 \amalg X_2$ then the diagram commutes
$$
\begin{CD}
C(X)                  @>{(id \times f)^{*}}>>            C(X_1) \times C(X_2)      \\
@V{N_{X/Y}}VV                              @VV{N_{X_1/Y}N_{X_2/Y}}V    \\
C(Y)                  @>{id}>>            C(Y)
\end{CD}
$$
\item[(iii')]
normalization: if $X=Y$ and the map $X \to Y$ is the identity then $N_{X/Y}=id_{C(X)}$.
\end{itemize}

\section{Unramified elements}
\label{SectUnramifiedElements}
Let $k$ be {\bf a field}, $\mathcal O$ be the $k$-algebra from Theorem
\ref{Aus_Buksbaum}
and $K$ be the fraction field of $\mathcal O$.
Let
$\mu: G\to C$
be the morphism of reductive $\mathcal O$-group schemes from Theorem
\ref{Aus_Buksbaum}.
We work in this section with {\it the category of commutative Noetherian $\mathcal O$-algebras}.
For a commutative $\mathcal O$-algebra $S$ set
\begin{equation}
\label{MainFunctor}
{\cal F}(S)=C(S)/\mu(G(S)).
\end{equation}
Let $S$ be an $\mathcal O$-algebra which is a domain and
let $L$ be its fraction field.
Define the {\it subgroup of $S$-unramified elements of $\mathcal F (L)$} as
\begin{equation}
\label{DefnUnramified}
\mathcal F_{nr,S}(L)=
\bigcap_{\mathfrak p \in Spec(S)^{(1)}} Im [ \mathcal F(S_{\mathfrak p})\to\mathcal F(L) ],
\end{equation}
where $Spec(S)^{(1)}$ is the set of hight $1$ prime ideals in $S$.
Obviously the image of $\mathcal F(S)$ in $\mathcal F(L)$ is contained in
$\mathcal F_{nr,S}(L)$. In most cases
$\mathcal F(S_{\mathfrak p})$
injects into
$\mathcal F(L)$
and
$\mathcal F_{nr,S}(L)$
is simply the intersection of all
$\mathcal F(S_{\mathfrak p})$.
For an element $\alpha \in C(S)$ we will write $\bar \alpha$ for its
image in ${\cal F}(S)$. {\it In this section we will write  ${\cal F}$
for the functor }
(\ref{MainFunctor}).
\begin{thm}[\cite{Ni}] 
\label{NisnevichCor}
Let $S$ be a $\mathcal O$-algebra which is discrete valuation ring with fraction field $L$.
The map
${\cal F}(S) \to {\cal F}(L)$
is injective.
\end{thm}

\begin{lem}
\label{BoundaryInj}
Let $\mu: G \to C$ be the above morphism of our reductive group schemes.
Let $H= \ker (\mu)$.
Then for an $\mathcal O$-algebra $L$, where $L$ is a field, the boundary map
$\partial: C(L)/{\mu (G(L))} \to \textrm{H}^1_{\text{\'et}}(L,H)$
is injective.
\end{lem}

Let $k$, $\mathcal O$ and $K$ be as above in this Section.
Let $\mathcal K$ be a field containing $K$ and
$x: \mathcal K^* \to \mathbb Z$
be a discrete valuation vanishing on $K$.
Let $A_x$ be the valuation ring of $x$. Clearly,
$\mathcal O \subset A_x$.
Let
$\hat A_x$ and $\hat {\mathcal K}_x$
be the completions of $A_x$ and $\mathcal K$ with respect to $x$.
Let
$i:\mathcal K \hookrightarrow \hat {\mathcal K}_x$
be the inclusion. By Theorem
\ref{NisnevichCor}
the map
${\cal F}(\hat A_x)\to {\cal F}(\hat{\mathcal K}_x)$
is injective. We will identify
${\cal F}(\hat A_x)$
with its image under this map. Set
$$
{\cal F}_x(\mathcal K)=i_*^{-1}({\cal F}(\hat A_x)).
$$

The inclusion
$A_x\hookrightarrow \mathcal K$
induces a map
$
{\cal F}(A_x) \to {\cal F}(\mathcal K)
$
which is injective by Theorem
\ref{NisnevichCor}.
So both groups
${\cal F}(A_x)$ and ${\cal F}_x(\mathcal K)$
are subgroups of
${\cal F}(\mathcal K)$.
The following lemma shows that
${\cal F}_x(\mathcal K)$
coincides with the subgroup of $
{\cal F}(\mathcal K)$
consisting of all elements {\it unramified} at $x$.

\begin{lem}
\label{TwoUnramified}
${\cal F}(A_x)={\cal F}_x(\mathcal K)$.
\end{lem}

Let $k$, $\mathcal O$ and $K$ be as above in this Section.
\begin{lem}
\label{KeyUnramifiedness}
Let $B \subset A$ be a finite extension of $K$-smooth algebras, which are domains
and each has dimension one.
Let $0 \neq f \in A$ and
let $h \in B\cap fA$ be such that the induced map
$B/hB\to A/fA$ is an isomorphism.
Suppose
$hA=fA\cap J^{\prime\prime}$
for an ideal
$J^{\prime\prime} \subseteq A$
co-prime to the ideal $fA$.

Let $E$ and $F$ be the field of fractions of $B$ and $A$ respectively.
Let $\alpha \in C(A_f)$ be such that
$\bar \alpha \in {\cal F}(F)$
is $A$-unramified. Then, for
$\beta= N_{F/E}(\alpha)$,
the class
$\bar \beta \in {\cal F}(E)$
is $B$-unramified.
\end{lem}

\begin{proof}
The only primes at which $\bar \alpha$ could be ramified are those which divide
$hA$.
Let $\mathfrak p$ be one of them. Check that $\bar \alpha$ is
unramified at $\mathfrak p$.

To do this we consider all primes
$\mathfrak q_1, \mathfrak q_2, \dots, \mathfrak q_n$
in $A$ lying over
$\mathfrak p$.
Let
$\mathfrak q_1$
be the unique prime dividing $f$ and lying over
$\mathfrak p$.
Then
$$
A\otimes_B  \hat B_{\mathfrak p} = \hat A_{\mathfrak q_1} \times \prod_{i \neq 1} \hat A_{\mathfrak q_i}
$$
with
$\hat A_{\mathfrak q_1}=\hat B_{\mathfrak p}$.
If $F$, $E$ are the fields of fractions of $A$ and $B$
then
$$
F \otimes_{B} \hat B_{\mathfrak p}=\hat F_{\mathfrak q_1} \times \prod_{i \neq 1}  \hat F_{\mathfrak q_n}
$$
and
$\hat F_{\mathfrak q_1}=\hat E_{\mathfrak p}$.
We will write
$\hat F_i$ for $\hat F_{\mathfrak q_i}$
and
$\hat A_i$ for $\hat A_{\mathfrak q_i}$.
Let
$$\alpha \otimes 1= (\alpha_1,\dots, \alpha_n)
\in C(\hat F_1) \times \dots \times  C(\hat F_n).$$
Clearly for $i \geq 2$ one has
$\alpha_i \in C(\hat A_i)$
and
$\alpha_1=\mu(\gamma_1)\alpha^{\prime}_1$
with
$\alpha^{\prime}_1 \in C(\hat A_1)=C(\hat B_{\mathfrak p})$
and
$\gamma_1 \in G(\hat F_1)=G(\hat E_{\mathfrak p})$.
Now
$\beta \otimes 1 \in C(\hat E_{\mathfrak p})$
coincides with the product
$$
\alpha_1N_{\hat F_2/\hat E_{\mathfrak p}}(\alpha_2)\cdots N_{\hat F_n/\hat E_{\mathfrak p}}(\alpha_n)=
\mu(\gamma_1)
[\alpha^{\prime}_1N_{\hat F_2/\hat E_{\mathfrak p}}(\alpha_2)\cdots N_{\hat F_n/\hat E_{\mathfrak p}}(\alpha_n)].
$$
Thus
$\overline {\beta \otimes 1}=\bar \alpha^{\prime}_1\overline {N_{\hat F_2/\hat E_{\mathfrak p}}(\alpha_2)}
\cdots \overline {N_{\hat F_n/\hat E_{\mathfrak p}}(\beta_n)} \in {\cal F}(\hat B_{\mathfrak p})$.
Let
$i: E \hra \hat E_{\mathfrak p}$
be the inclusion and
$i_*:{\cal F}(E) \to {\cal F}(\hat E_{\mathfrak p})$
be the induced map.
Clearly
$i_*(\bar \beta)=\overline {\beta \otimes 1}$
in
${\cal F}(\hat E_{\mathfrak p})$.
Now
Lemma
\ref{TwoUnramified}
shows that
the element
$\bar \beta \in {\cal F}(E)$
belongs to
${\cal F}(B_{\mathfrak p})$.
Hence $\bar \beta$ is $B$-unramified.

\end{proof}

\section{Specialization maps}
\label{SectSpecializationMaps}
Let $k$ be {\bf a field}, $\mathcal O$ be the $k$-algebra from Theorem
\ref{Aus_Buksbaum}
and $K$ be the fraction field of $\mathcal O$.
Let
$\mu: G\to C$
be the morphism of reductive $\mathcal O$-group schemes from Theorem
\ref{Aus_Buksbaum}.
We work in this section with {\it the category of commutative $K$-algebras} and with
the functor
\begin{equation}
\label{MainFunctor2}
\mathcal F: S\mapsto C(S)/\mu(G(S))
\end{equation}
defined on the category of $K$-algebras. So, we assume in this Section that
each ring from this Section is equipped with a distinguished $K$-algebra structure and
each ring homomorphism from this Section respects that structures.
Let $S$ be an $K$-algebra which is a domain and
let $L$ be its fraction field.
Define the {\it subgroup of $S$-unramified elements $\mathcal F_{nr,S}(L)$ of $\mathcal F (L)$} by formulae
(\ref{DefnUnramified}).

For a regular domain $S$ with the fraction field $\cal K$
and each height one prime $\mathfrak p$ in $S$
we construct {\bf specialization maps}
$s_{\mathfrak p}: {\cal F}_{nr, S}({\cal K}) \to {\cal F} {(K(\mathfrak p))}$,
where $\cal K$ is the field of fractions of $S$ and
$K(\mathfrak p)$
is the residue field of
$R$ at the prime $\mathfrak p$.

\begin{defn}
\label{SpecializationDef}
Let
$Ev_{\mathfrak p}: C(S_{\mathfrak p}) \to C(K(\mathfrak p))$
and
$ev_{\mathfrak p}: {\cal F}(S_{\mathfrak p}) \to {\cal F}(K(\mathfrak p))$
be the maps induced by the canonical $K$-algebra homomorphism
$S_{\mathfrak p} \to K(\mathfrak p)$.
Define a homomorphism
$s_{\mathfrak p}: {\cal F}_{nr, S}({\cal K}) \to {\cal F} {(K(\mathfrak p))}$
by
$s_{\mathfrak p}(\alpha)= ev_{\mathfrak p}(\tilde \alpha)$,
where
$\tilde \alpha$
is a lift of $\alpha$ to
${\cal F}(S_{\mathfrak p})$.
Theorem
\ref{NisnevichCor}
shows that the map $s_{\mathfrak p}$ is well-defined.
It is called the specialization map. The map $ev_{\mathfrak p}$ is called the evaluation
map at the prime $\mathfrak p$.

Obviously for
$\alpha \in C(S_\mathfrak p)$
one has
$s_{\mathfrak p}(\bar \alpha)=\overline {Ev_{\mathfrak p}(\alpha)}
\in {\cal F}(K(\mathfrak p))$.
\end{defn}

We need the following theorem and its corollary.
\begin{thm}[Homotopy invariance]
\label{HomInvNonram}
Let $S \mapsto {\cal F}(S)$ be the functor defined by the formulae
(\ref{MainFunctor2})
and let
${\cal F}_{nr,K[t]}(K(t))$ be defined by the formulae
(\ref{DefnUnramified}).
Let
$K(t)$ be the rational function field in one variable.
Then one has
$$
{\cal F}(K)={\cal F}_{nr,K[t]}(K(t)).
$$
\end{thm}

\begin{cor}
\label{TwoSpecializations}
Let $S \mapsto {\cal F}(S)$ be the functor defined in
(\ref{MainFunctor}).
Let
$$
s_0, s_1: {\cal F}_{nr, K[t]}(K(t)) \to {\cal F}(K)
$$
be the specialization maps at zero and at one
(at the primes (t) and (t-1)). Then $s_0=s_1$.
\end{cor}

\section{A purity theorem}
\label{Section_Aus_Buksbaum}
Let $k$ be a field. The main result of the present section is the following
\begin{thm}
\label{Aus_Buksbaum}
Let $\mathcal O$ be the semi-local ring of finitely many closed points
on a $k$-smooth irreducible affine $k$-variety $X$.
Let $K=k(X)$.
Let
$$\mu: \bG \to \mathbf C$$
be a smooth $\mathcal O$-morphism of reductive
$\mathcal O$-group schemes, with a torus $\mathbf C$.
Suppose additionally that
the kernel of $\mu$ is a reductive $\mathcal O$-group scheme.
Then the following sequence
\begin{equation}
\label{Aus_Buks_sequence}
\mathbf C(\mathcal O)/\mu(\bG(\mathcal O)) \to
\mathbf C(K)/\mu(\bG(K)) \xrightarrow{\sum res_{\mathfrak p}} \bigoplus_{\mathfrak p}
\mathbf C(K)/[\mathbf C(\mathcal O_{\mathfrak p})\cdot \mu(\bG(K))]
\end{equation}
is exact in the middle term, where $\mathfrak p$ runs over
the height $1$ primes of $\mathcal O$ and $res_{\mathfrak p}$ is the natural map (the projection to the factor group).
\end{thm}


\begin{defn}
Let $a\in \mathbf C(k(X))$. Its class
$\bar a \in \bar {\mathbf C}(k(X)):=\mathbf C(K)/\mu(\bG(K))$
is called unramified at a height $1$ prime ideal $\mathfrak p$ of $k[X]$,
if the element $\bar a$ is in the image of the group
$\bar{\mathbf C}(\mathcal O_{\mathfrak p})$.
Let
$S\subset k[X]$ be a multiplicative system.
The class
$\bar a \in \bar {\mathbf C}(k(X))$
is called $k[X]_S$~-unramified, if it is unramified at any codimension one prime ideal of $k[X]_S$.
{\bf Particularly, the class $\bar a \in \bar {\mathbf C}(k(X))$ is called $X$-unramified,
if it is unramified at any codimension one prime ideal of $k[X]$.
}
\end{defn}
The following lemma is obvious.
\begin{lem}
\label{UnramifiedToUnramified}
Let $\phi: Y\to X$ be a smooth morphism of smooth irreducible affine $k$-varieties.
This morphism induces an obvious map
$\bar \varphi^*: \bar{\mathbf C}(k(X))\to \bar {\mathbf C}(k(Y))$,
which takes $X$-unramified elements to $Y$-unramified elements.
If $S\subset k[Y]$ be a multiplicative system, then
the homomorphism
$\bar \varphi^*$
takes $X$-unramified elements to $k[Y]_S$~-unramified elements.
\end{lem}

\begin{proof}[Proof of Theorem \ref{Aus_Buksbaum}]
Assume firstly that $\mu$ is "constant", i.e
there are a reductive group $\bG_0$, a torus $\mathbf C_0$ over the field $k$ and an algebraic $k$-group morphism $\mu_0$
and $U$-group schemes isomorphisms
$$\Phi: \bG_{0,U}=\bG_0\times_{Spec(k)} U \to \bG \ \ \text{and} \ \ \Psi: \mathbf C_{0,U}=\mathbf C_0\times_{Spec(k)} U \to \mathbf C$$
such that
$\Psi \circ \mu_{0,U}=\mu\circ \Phi$.

Let $a_K \in \mathbf C(k(X))$ be such that its class in $\bar {\mathbf C}(K)$ is $\mathcal O$-unramified.
Then there is a non-zero function $\text{f}\in k[X]$ such that the element $a_K$
is defined over $X_{\text{f}}$, that is
there is given an element
$a \in \mathbf C(k[X_{\text{f}}])$
for a non-zero function
$\text{f} \in k[X]$
such that
the image of $a$ in $\mathbf C(K)$ coincides with the element $a_K$.
Shrinking $X$ we may assume further that $\text{f}$ vanishes at each $x_i$'s and
the $k$-algebra $k[X]/(\textrm{f})$ {\bf is reduced}.
Shrinking $X$ once again we may and will assume also that
$\bar a \in \bar {\mathbf C}(k[X]_\textrm{f})$
is $k[X]$-unramified.
By Theorem
\ref{equating3}
there is a diagram of the form
(\ref{DeformationDiagram0})
together with
with the scheme $\mathcal X$, the morphisms $q_U$, $\sigma$ and $q_X$,
and the function $f^{\prime} \in q^*_X(\textrm{f} \ )k[\mathcal X]$,
which enjoys the properties $(a)$ to $(g)$ from that Theorem.
From now on and till the end of this proof
we will use the notation from Theorem \ref{equating3}.

The morphism $\sigma$ from that theorem
is finite surjective and the schemes $\Aff^1_U$ and $\mathcal X$ are regular.
Thus by a theorem of Grothendieck
\cite[Thm. 18.17]{E}
the morphism $\sigma$ is flat and finite. Thus any base change of $\sigma$ is finite and flat.
Set $\alpha:=q^*_f(a)\in \mathbf C(\mathcal X_f)$
where $q_f: \mathcal X_f \to X_{\textrm{f}}$ is the restriction of $f$ to $\mathcal X$ and set
\begin{equation}
\label{U_alpha}
a_U:=N_{\mathcal D_1/U}(\alpha|_{\mathcal D_1})\cdot N_{\mathcal D^{\prime}_0/U}(\alpha|_{\mathcal D^{\prime}_0})^{-1} \in \mathbf C(U).
\end{equation}

\begin{clm}
Let $\eta_U: Spec(k(X)) \to Spec(\mathcal O)=U$ and $\eta: Spec(k(X)) \to X_{\textrm{f}}$
be the generic points of $U$ and $X_{\textrm{f}}$
respectively. Then
$\eta^*_U(\bar {a_U})=\eta^*(\bar a) \in \bar {\mathbf C}(k(X))$.
\end{clm}
Since $\eta^*(a)=a_K$, this Claim completes the proof of Theorem \ref{Aus_Buksbaum} in the constant case.
To prove the Claim
consider the scheme $\mathcal X$ and its closed and open subschemes as $U$-schemes via
the morphism $q_U$. Set $K=k(X)$. Taking the base change of
$\mathcal X$, $\Aff^1_U$ and $\sigma$ via the
morphism
$\eta_U: Spec(K) \to U$
we get a morphism of the $K$-shemes
$\Aff^1_{K} \xleftarrow{\sigma_{K}} \mathcal X_{K}$.
Recall that
the class
$\bar a \in \bar {\mathbf C}(X_{\textrm{f}})$ is $X$-unramified.
By Lemma
\ref{UnramifiedToUnramified}
the class
$\bar \alpha\in \bar {\mathbf C}(\mathcal X_f)$
is
$\mathcal X$-unramified.
Hence its image $\bar \alpha_{K}$ in $\bar {\mathbf C}(K(\mathcal X_K))$
is $X_K$-unramified too.
The items $(ii)$, $(v)$ of Corollary
\ref{ElementaryNisSquareNew_1}
and Lemma
\ref{KeyUnramifiedness} show that for the element
$\beta_t:=N_{K(\mathcal X_K)/K(\Aff^1_K)}(\alpha_K) \in \mathbf C(K(t))$
the class $\bar \beta_t \in \bar {\mathbf C}(K(t))$
is $\Aff^1_K$-unramified.
By Theorem
\ref{HomInvNonram}
the class $\bar \beta_t$ is constant,
t.e. it comes from the field $K$.
By Corollary
\ref{TwoSpecializations}
its specializations at the $K$-points
$0$ and $1$ of the affine line $\Aff^1_K$ coincide:
$s_0(\bar \beta_t)= s_1(\bar \beta_t) \in \bar {\mathbf C}(K)$.
The properties
$(d)$,$(c)$ and $(e)$ and the equality $q_X\circ \Delta=can$
from Theorem
\ref{equating3}
show that
$\mathcal D_{1,K}, \mathcal D^{\prime}_{0,K}, \Delta(Spec(K)) \subset (\mathcal X_f)_K$.
Thus there is a Zariski open neighborhood $V$ of the $K$-points $0$ and $1$ in $\Aff^1_K$
such that
$W:=(\sigma_K)^{-1}(V) \subset (\mathcal X_f)_K$.
Hence for
$\beta_V:=N_{W/V}(\alpha|_W)$,
one has the equality $\beta_V=\beta_t$ in $\mathbf C(K(t))$.
Thus one has equalities
$$\overline {\beta(1)}=s_1(\bar \beta_t)=s_0(\bar \beta_t)=\overline {\beta(0)}$$
(see the remark at the end of Definition \ref{SpecializationDef}).
By the properties $(i')$, $(ii')$ and $(iii')$
of the norm maps (see Section \ref{SectNorms}) one has equalities
$$N_{\mathcal D_{1,K}/K}(\alpha|_{\mathcal D_{1,K}})= \beta(1) \ \ \text{and} \ \  \beta(0)=N_{\mathcal D_{0,K}/K}(\alpha|_{\mathcal D_{0,K}})=
N_{\mathcal D^{\prime}_{0,K}/K}(\alpha|_{\mathcal D^{\prime}_{0,K}}) \cdot s^*_K(\alpha_K)
$$
By the base change property of the norm maps one has the equality
$$\eta^*_U(a_U)=N_{\mathcal D_{1,K}/K}(\alpha|_{\mathcal D_{1,K}})\cdot [N_{\mathcal D^{\prime}_{0,K}/K}(\alpha|_{\mathcal D^{\prime}_{0,K}})]^{-1}$$
Hence $s^*_K(\bar \alpha_K)=\eta^*_U(\bar a_U)$ in $\bar {\mathbf C}(k(X))$.
Finally, the composite map
$Spec(K)\xrightarrow{\Delta_K} (\mathcal X_f)_K \to \mathcal X_f \xrightarrow{q_f} X_{\textrm{f}}$
coincides with the canonical map
$\eta: Spec(K) \to X_{\textrm{f}}$.
Hence
$s^*_K(\bar \alpha_K)=\eta^*(\bar a)$, which proves the Claim.
Whence the Theorem \ref{Aus_Buksbaum} in the constant case.

{\it In the general case} there are two functors on the category of $\mathcal X$-schemes. Namely,
$\bar C$
and
${_{U}}\bar C$.
If $r: \mathcal Y \to \mathcal X$ is a scheme morphism, then
$\bar C(\mathcal Y):=C(\mathcal Y)/(\mu(G(\mathcal Y)))$
and
${_{U}}\bar C(\mathcal Y):= {_{U}}C(\mathcal Y)/(\mu({_{U}}G(\mathcal Y)))$.
Here $\mathcal Y$ is regarded as an $X$-scheme via the morphism
$q_X\circ r$ and is regarded as an $U$-scheme via the morphism
$q_U\circ r$. The $\mathcal X$-group scheme isomorphisms
$\Phi$ and $\Psi$ from Theorem
\ref{equating3}
induce a group isomorphism
$$\bar \Psi_{\mathcal Y}: {_{U}}\bar C(\mathcal Y) \to \bar C(\mathcal Y),$$
which respect to $\mathcal X$-schemes morphisms.
Moreover, if the scheme $U$ is regarded as an $\mathcal X$-scheme via the morphism
$\Delta$, then the isomorphism $\bar \Psi_{\mathcal Y}$ is the identity.
And similarly for any $U$-scheme $g: W \to U$ regarded as an $\mathcal X$-scheme via
the morphism $\Delta\circ g$ the the isomorphism $\bar \Psi_{W}$ is the identity.

Set $\alpha:=q^*_f(a)\in \mathbf C(\mathcal X_f)$
where $q_f: \mathcal X_f \to X_{\textrm{f}}$ is above in this proof. Let
$_{_{U}}\alpha \in {_{U}}C(\mathcal X)$
be a unique element such that
$\bar \Psi_{\mathcal X}(_{_{U}}\alpha)=\alpha$.
Set
\begin{equation}
\label{U_alpha_2}
_{_{U}}a:=N_{\mathcal D_1/U}((_{_{U}}\alpha)|_{\mathcal D_1})\cdot N_{\mathcal D^{\prime}_0/U}((_{_{U}}\alpha)|_{\mathcal D^{\prime}_0})^{-1}
\in \ _{_{U}}\mathbf C(U) \ \ \text{and} \ \ a_U:=\Psi_U(_{_{U}}a) \in \mathbf C(U)
\end{equation}
We left to the reader to proof the following Claim
\begin{clm}
Let $\eta_U: Spec(k(X)) \to Spec(\mathcal O)=U$ and $\eta: Spec(k(X)) \to X_{\textrm{f}}$
be as above in this proof. Then
$$\eta^*_U(\bar {a_U})=\eta^*(\bar a) \in \bar {\mathbf C}(k(X))$$
\end{clm}
Since $\eta^*(a)=a_K$, this Claim completes the proof of Theorem \ref{Aus_Buksbaum}.

\end{proof}

\section{One more purity theorem}
\label{OneMorePurity}
The main result of this section is theorem \ref{PurityForSubgroup}.
Let $k$, $\mathcal O$ and $K$ be as in Theorem
\ref{Aus_Buksbaum}.
Let $G$ be a semi-simple $\mathcal O$-group scheme.
Let
$i: Z \hra G$ be a closed subgroup scheme of the center $Cent(G)$.
{\bf It is known that $Z$ is of multiplicative type}.
Let $G'=G/Z$ be the factor group,
$\pi: G \to G'$ be the projection.
It is known that $\pi$ is finite surjective and faithfully flat. Thus
the sequence of $\mathcal O$-group schemes
\begin{equation}
\label{ZandGndGprime}
\{1\} \to Z \xra{i} G \xra{\pi} G^{\prime} \to \{1\}
\end{equation}
induces an exact sequence of group sheaves in $\text{fppt}$-topology.
Thus for every $\mathcal O$-algebra $R$ the sequence
(\ref{ZandGndGprime})
gives rise to a boundary operator
\begin{equation}
\label{boundary}
\delta_{\pi,R}: G'(R) \to \textrm{H}^1_{\text{fppf}}(R,Z)
\end{equation}
One can check that it is a group homomorphism
(compare \cite[Ch.II, \S 5.6, Cor.2]{Se}).
Set
\begin{equation}
\label{AnotherFunctor}
{\cal F}(R)= \textrm{H}^1_{\text{fttf}}(R,Z)/ Im(\delta_{\pi,R}).
\end{equation}
Clearly we get a functor on the category of $\mathcal O$-algebras.
\begin{thm}
\label{PurityForSubgroup}
Let $\mathcal O$ be the semi-local ring of finitely many {\bf closed} points
on a $k$-smooth irreducible {\bf affine} $k$-variety $X$.
Let $G$ be a semi-simple $\mathcal O$-group scheme.
Let
$i: Z \hra G$ be a closed subgroup scheme of the center $Cent(G)$.
Let
${\cal F}$
be the functor
on the category
$\mathcal O$-algebras
given by
(\ref{AnotherFunctor}).
Then the sequence
\begin{equation}
\label{Aus_Buks_sequence_2}
{\cal F}(\mathcal O) \to
{\cal F}(K) \xrightarrow{\sum can_{\mathfrak p}} \bigoplus_{\mathfrak p}
{\cal F}(K)/Im[{\cal F}(\mathcal O_{\mathfrak p}) \to {\cal F}(K)]
\end{equation}
is exact, where $\mathfrak p$ runs over
the height $1$ primes of $\mathcal O$ and $can_{\mathfrak p}$ is the natural map (the projection to the factor group).
\end{thm}


\begin{proof}[Proof of Theorem \ref{PurityForSubgroup}]
The group $Z$ is of  multiplicative type.
So we can find a finite \'{e}tale $\mathcal O$-algebra $A$ and
a closed embedding
$Z \hra R_{A/{\mathcal O}}(\mathbb G_{m,\ A})$
into the permutation torus
$T^{+}=R_{A/\mathcal O}(\mathbb G_{m,\ A})$.
Let
$G^{+}=(G \times T^{+})/Z$
and
$T=T^{+}/Z$,
where
$Z$ is embedded in
$G \times T^{+}$
diagonally.
Clearly
$G^{+}/G=T$.
Consider a commutative diagram
$$
\xymatrix {
{}           & \{1\}                       & \{1\} \\
{}           & {G'} \ar[r]^{id}  \ar[u]          & {G'}   \ar[u]          \\
\{1\} \ar[r] & {G} \ar[r]^{j^+} \ar[u]^{\pi} & {G^+} \ar[r]^{\mu^+} \ar[u]^{\pi^+} & {T} \ar[r] & \{1\} \\
\{1\} \ar[r] & {Z} \ar[r]^{j} \ar[u]^{i} & {T^+} \ar[r]^{\mu} \ar[u]^{i^+} & {T} \ar[u]_{id} \ar[r] &  \{1\} \\
{}           & \{1\} \ar[u]                      & \{1\} \ar[u]\\
}
$$
with exact rows and columns.
By \cite[Thm.11.7]{Gr3}
and Hilbert 90
for the semi-local $\mathcal O$-algebra $A$ one has
$\textrm{H}^1_{\text{fppf}}(\mathcal O,T^+)= \textrm{H}^1_{\text{\'et}}(\mathcal O,T^+)= \textrm{H}^1_{\text{\'et}}(A,\Bbb G_{m,A})=\{* \}$.
So, the latter diagram gives rise to
a commutative diagram of pointed sets
$$
\xymatrix {
{}           & {}                            & {\textrm{H}^1_{\text{fppt}}(\mathcal O,G')} \ar[r]^{id}            & {\textrm{H}^1_{\text{fppt}}(\mathcal O,G')}            \\
{G^+(\mathcal O)} \ar[r]^{\mu^+_\mathcal O} & {T(\mathcal O)} \ar[r]^{\delta^+_\mathcal O}  & {\textrm{H}^1_{\text{fppt}}(\mathcal O,G)} \ar[r]^{j^+_*} \ar[u]^{\pi_*} & {\textrm{H}^1_{\text{fppt}}(\mathcal O,G^+)} \ar[u]^{\pi^+_*}  \\
{T^+(\mathcal O)} \ar[r]^{\mu_\mathcal O} \ar[u]^{i^+_*} & {T(\mathcal O)} \ar[r]^{\delta_\mathcal O} \ar[u]^{id} & {\textrm{H}^1_{\text{fppt}}(\mathcal O,Z)} \ar[r]^{\mu} \ar[u]^{i_*} & {\{*\}} \ar[u]^{i^+_*}  \\
{}           & {} & {G'(\mathcal O)} \ar[u]^{\delta_{\pi}}                      & {} \\
}
$$
with exact rows and columns. It follows that
$\pi^+_*$ has  trivial kernel and one has a chain of group isomorphisms
\begin{equation}
\label{F=T/mu_+(G_+)}
\textrm{H}^1_{\text{fppf}}(\mathcal O,Z) / Im (\delta_{\pi,\mathcal O})= ker (\pi_*)= ker (j^+_*) = T(\mathcal O)/\mu^+ (G^+(\mathcal O)).
\end{equation}
Clearly these isomorphisms respect $\mathcal O$-homomorphisms of semi-local $\mathcal O$-algebras.

The morphism $\mu^{+}: G^{+} \to T$ is a smooth $\mathcal O$-morphism of reductive
$\mathcal O$-group schemes, with the torus $T$. The kernel
$ker(\mu^{+})$
is equal to $G$ and $G$ is a reductive $\mathcal O$-group scheme.
Now by Theorem
\ref{Aus_Buksbaum}
the sequence
(\ref{Aus_Buks_sequence})
is exact.
Thus the sequence
(\ref{Aus_Buks_sequence_2})
is exact too.

\end{proof}

\section{Proof of Theorem \ref{MainThmGeometric}}\label{sec:MainThmGeometric}
\begin{proof}[Proof of the semi-simple case of Theorem \ref{MainThmGeometric}]
Let $\mathcal O$ and $G$ be the same as in Theorem \ref{MainThmGeometric} and assume additionally that $G$ is semi-simple.
We need to prove that
\begin{equation}
\label{Gr_SerreForGsemisimple}
ker[H^1_{\text{\'et}}(\mathcal O,G) \to H^1_{\text{\'et}}(K,G)]=* .
\end{equation}
Let
$G^{sc}_{}$
be the corresponding simply-connected semi-simple $\mathcal O$-group scheme.
\begin{clm}
\label{simplyconnectedsemisimpleinjective}
Under the hypotheses of Theorem \ref{MainThmGeometric} for all semi-simple reductive $\mathcal O$-group scheme $G$
the map
$H^1_{\text{\'et}}(\mathcal O,G^{sc}) \to H^1_{\text{\'et}}(K,G^{sc})$
is injective.
\end{clm}
In fact, let $\xi, \zeta \in H^1_{\text{\'et}}(\mathcal O,G^{sc})$ be two elements such that its images $\xi_K, \zeta_K$ in
$H^1_{\text{\'et}}(K,G^{sc})$
are equal. Let $_{\xi}G^{sc}$, $_{\zeta}G^{sc}$ be the corresponding principal $G^{sc}$-bundles over $\mathcal O$ and $G^{sc}(\zeta)$ be
the inner form of the $\mathcal O$-group scheme $G^{sc}$ corresponding to $\zeta$. The $\mathcal O$-scheme
$\underline {Iso}(_{\xi}G^{sc}, \  _{\zeta}G^{sc})$
is a principal
$G^{sc}(\zeta)$-bundle over $\mathcal O$,
which is trivial over $K$. Since $G^{sc}(\zeta)$ is simly-connected semi-simple reductive over $\mathcal O$,
the $\mathcal O$-scheme
$\underline {Iso}(_{\xi}G^{sc}, \ _{\zeta}G^{sc})$
has an $\mathcal O$-point by the hypotheses of
Theorem \ref{MainThmGeometric}. Whence the Claim.

To finish the proof of the semi-simple case of Theorem \ref{MainThmGeometric}
it remains to repeat literally
the arguments from the proof of the semi-simple case of
\cite[Sect. 11, Proof of Thm. 1.0.3]{Pa2}.

To continue the proof of Theorem \ref{MainThmGeometric}
we need the following claim, which is proved exactly
as the previous one.
\begin{clm}
\label{semisimpleinjective}
Under the hypotheses of Theorem \ref{MainThmGeometric} for all semi-simple reductive $\mathcal O$-group scheme $G$
the map
$H^1_{\text{\'et}}(\mathcal O,G) \to H^1_{\text{\'et}}(K,G)$
is injective.
\end{clm}

{\it The end of the proof of Theorem \ref{MainThmGeometric}}.
Just repeat literally the arguments from the
respecting part of the proof of
\cite[Section 11]{Pa2}.
\end{proof}

\end{document}